\documentclass[a4paper,12pt]{article}

\usepackage[utf8]{inputenc}
\usepackage{amsthm}
\usepackage[english]{babel}
\usepackage{amsmath, amsfonts, amssymb}
\usepackage{cite}
\usepackage{graphicx}
\usepackage{enumitem}
\usepackage{hyperref}
\usepackage{tikz-cd}
\usepackage{mathtools}
\usepackage{geometry}
\usepackage{mathptmx} % - fonte
\usepackage{extarrows}
\newtheorem{theorem}{Theorem}[section]
\newtheorem{lemma}[theorem]{Lemma}

\newtheorem{proposition}[theorem]{Proposition}

\newtheorem{corollary}[theorem]{Corollary}

\newtheorem{definition}[theorem]{Definition}

\newtheorem*{remark}{Remark}
\newtheorem{example}{Example}
\newtheorem*{teo}{Main Theorem}
\newtheorem*{lemma*}{Lemma}

\newcommand{\Hom}{\textnormal{Hom}}
\newcommand{\Ext}{\textnormal{Ext}}
\newcommand{\ext}{\mathcal{E} \mathit{xt}}
\newcommand{\Nf}{N_{\mathcal{F}}}
\newcommand{\Ng}{N_{\mathcal{G}}}
\newcommand{\Tf}{T_\mathcal{F}}
\newcommand{\Tg}{T_\mathcal{G}}
\newcommand{\Z}{{\mathbb Z}}

\newcommand{\C}{\mathbb{C}}

\newcommand{\p}{{\mathbb{P}^{3}}}
\newcommand{\pn}{{\mathbb{P}^n}}

\newcommand{\op}{{\mathcal O}_{\mathbb{P}^{3}}}
\newcommand{\opn}{{\mathcal O}_{\mathbb{P}^n}}
\newcommand{\ox}{{\mathcal O}_{X}}

\newcommand{\tpb}{{\rm T}{\mathbb{P}^{3}}}
\newcommand\restr[2]{{% we make the whole thing an ordinary symbol
  \left.\kern-\nulldelimiterspace % automatically resize the bar with \right
  #1 % the function
  \littletaller % pretend it's a little taller at normal size
  \right|_{#2} % this is the delimiter
  }}

\title{Distributions with Unstable Tangent Sheaf on \(\p\)}
\date{}
\author{Pedro Pfarrius Barbassa}
\begin{document}
\maketitle

\vspace{0.3cm}
\begin{abstract}
    We study codimension one distributions on the projective three-space, focusing on cases where the tangent sheaf of the distribution is nonsplit and unstable. We relate the order of nonstability to the degree of the induced subfoliation by curves, showing that the order of nonstability is bounded. Moreover, we classify the tangent sheaf of the codimension one distributions that admit a subfoliation by curves of degree 1. In other words, assuming the sheaf is nonsplit, we classify the situations in which the tangent sheaf attains the maximal possible order of nonstability.
\end{abstract}

\tableofcontents

\section{Introduction}
\paragraph{}
A distribution on a smooth complex variety \(X\) is a subsheaf of \(TX\) whose quotient is a torsion-free sheaf. The singularities of a distribution are the points where the quotient fails to be locally free. Among such objects are foliations, which are integrable distributions in the Frobenius sense. Foliation theory is a vast area of algebraic geometry with many different points of view; see \cite{loray2018singular,bogomolov2016rational,araujo2013fano,baum1972singularities}.

In the paper \cite{CCJ}, the authors proved several results concerning codimension one distributions and foliations on \(\p\), relating properties of the tangent sheaf to their singular scheme. They also provided an equivalent definition of families of distributions together with a fine moduli space, which was compared with the one developed in \cite{Federico}. An important observation is that the moduli space of distributions does not require stability of the tangent sheaf. They further established conditions under which a distribution cannot be a foliation. These obstructions were later developed in a more general setting in \cite{calvoandrade2025connectednesssingularsetholomorphic}, where it was shown that, under certain conditions, the higher dimensional component of the singular scheme of a foliation is connected, and foliations on the projective space satisfies those conditions.

Together with \cite{galeano2022codimension}, there is a complete classification of the tangent sheaf of codimension one distributions on \(\p\) of degree \(d \leq 2\). In these cases, the tangent sheaf is split or at least \(\mu\)-semistable, so unstable sheaves that are not split appear only in higher degrees. In fact, they appear naturally when one allows the degree of \(\mathcal{F}\) to be \(d \geq 3\) and the tangent sheaf admits a global section. In such a case, the tangent sheaf either splits or is unstable.

Unstable reflexive sheaves on \(\p\) were studied in \cite{sauer1984nonstable}, were the author defined a new invariant called the order of nonstability. We will apply this notion to the tangent sheaf of a codimension one distributions on \(\p\).

In this paper, we continue the study of codimension one distributions on \(\p\), focusing on the cases where the tangent sheaf is unstable. To measure stability, we consider the minimum value \(t_{\mathcal{F}}\) such that the tangent sheaf admits a global section, as in Definition \ref{menorseçao}. This value is also the minimum degree of a subfoliation by curves admitted by the distribution. As will be shown, such a subfoliation imposes numerical conditions on the codimension one distribution. We may correlate the instability of the tangent sheaf of a codimension one distribution with small values of \(t_{\mathcal{F}}\). In other words, we prove the following result.
\begin{lemma*}[\ref{lemmaintro}]
    Let \(\mathcal{F}\) be a codimension one distribution of degree \(d\) on \(\p\) such that \(\Tf\) is nonsplit. Then \(\Tf\) is unstable if and only if \(d \geq 3\) and
    \begin{equation*}
        1 \leq t_{\mathcal{F}} \leq \frac{d - 2 + \varepsilon}{2},
    \end{equation*}
    where \(\varepsilon \in \{0,1\}\) and \(\varepsilon \equiv d (\textnormal{mod }2)\). Moreover, \(\Tf\) will be unstable of order \(\frac{d+ \varepsilon}{2} - t_{\mathcal{F}}\). 
\end{lemma*}

Using the classification of foliations by curves of low degree, we obtain a complete classification of the tangent sheaf of every codimension one distribution with \(t_{\mathcal{F}} = 1\) in Proposition \ref{Tfcomseção}. In other words, together with Corollary \ref{tf=-1} we classify the tangent sheaf of all codimension one distributions such that \(H^0(\Tf) \neq 0\). When one allows degree \(d \geq 3\), such codimension one distributions have tangent sheaves that are split or unstable with the maximal possible order.

\begin{teo}[\ref{oteorema}]
    Let \(\mathcal{F}\) be a codimension one distribution of degree \(d\) on \(\p\) such that \(\Tf\) is nonsplit and unstable with maximal possible order. Then \(d \geq 3\) and
    \begin{enumerate}
        \item $\Tf$ has Chern classes $(2-d,1,d)$, generated by the extension of a line, \(\textnormal{Sing}_1(\mathcal{F})\) is a curve of degree \(d^2+1\) with genus \(p_a = d(d-1)^2\) and \(\textnormal{Sing}_0(\mathcal{F})\) consists of \(d\) collinear points; or

        \item $\Tf$ has Chern classes $(2-d,2,2d)$, generated by the extension of a pair of skew lines (or a double line of genus -1), \(\textnormal{Sing}_1(\mathcal{F})\) is a curve of degree \(d^2\) (complete intersection) and \(\textnormal{Sing}_0(\mathcal{F})\) consists of two pairs of \(d\) collinear points (or \(d\) collinear points on the double line).
    \end{enumerate}
\end{teo}

We show that every reflexive sheaf in item 1, generated by an extension of a line, is the tangent sheaf of a codimension one distribution. For those in item 2, we cannot guarantee such existence, but we give an explicit example of a homogeneous 1-form vanishing exactly at the prescribed singular scheme.

\paragraph{}
This paper is organized as follows. In Section \ref{section 2}, we review the basic definitions and results concerning distributions and the stability of  reflexive sheaves on \(\pn\). In Section \ref{section 3}, we focus on codimension one distributions on \(\p\), emphasizing how their subfoliations impose numerical conditions on the tangent sheaf. We also provide a classification of codimension one distributions that admit subfoliations by curves of degree 0 and 1. In Section \ref{section 4}, we relate these numerical conditions to the instability of the tangent sheaf.

\section{Preliminaries}\label{section 2}

\subsection{Distributions and Foliations on Smooth Varieties}
\paragraph{} A complete treatment of this subject can be found in \cite{CCJ}. Here, we will only state the definitions required to understand the proof of our main result. For the reader not familiar with basic concepts of algebraic geometry we recommend \cite{hartshorne2013algebraic}. In the following, $X$ is a smooth complex projective variety of dimension $n$.
\begin{definition}
    A (singular) distribution $\mathcal{F}$ of codimension $r$ on $X$ is given by an exact sequence
\begin{equation}\label{distemX}
    0 \longrightarrow \Tf \xlongrightarrow{\phi} TX \longrightarrow \Nf \longrightarrow 0,
\end{equation}
where $\Tf$ is a reflexive sheaf of rank $n-r$, and $\Nf$ is a torsion-free sheaf of rank $r$, called the tangent and normal sheaves of $\mathcal{F}$, respectively.
\end{definition}

The singular scheme of $\mathcal{F}$ will be the points where the stalk of the normal sheaf is not a free module over the local ring,   
\begin{equation*}
    \textnormal{Sing}(\mathcal{F}) = \textnormal{Sing}(\Nf) := \bigcup_{i=1}^{n-1} \textnormal{Supp}(\ext^i(\Nf,\ox)).
\end{equation*}
In other words, \(\textnormal{Sing}(\mathcal{F})\) are the points that \(\Nf\) fails to be a locally free. Since $\Nf$ is at least torsion-free, by \cite[Lemma 1.1.8.]{okonek1980vector} we have that $\textnormal{Sing}(\mathcal{F})$ is a subscheme of codimension at least 2 in \(X\).
\begin{remark}
    We say that a distribution \(\mathcal{F}\) is regular if it has no singularities, that is \(\textnormal{Sing}(\mathcal{F}) = \emptyset\). Notice that this is equivalent to the normal sheaf \(\Nf\) being locally free. 
\end{remark}

\paragraph{}
The morphism \(\phi: \Tf \hookrightarrow \tpb\) defined in (\ref{distemX}), by taking the maximal exterior power, generates a global section 
\begin{equation}\label{quasiespdist}
    \omega_{\mathcal{F}} \in H^0\big(X, \bigwedge^{n-r} TX \otimes \det(\Tf)^{\vee} \big) = H^0\big(X,\Omega^r_X \otimes \det(TX) \otimes \det(\Tf)^{\vee}  \big).
\end{equation}
We say that two distributions \(\mathcal{F}\) and \(\mathcal{F}'\) are isomorphic, if we have an isomorphism of sheaves \(\Tf \xrightarrow{\sim} T_{\mathcal{F}'}\) such that the diagramm comutes
\begin{equation*}
    \begin{tikzcd}
        \Tf \arrow[r, hook, "\phi"] \arrow[d, "\rotatebox{90}{$\sim$}"] & TX \arrow[d, equal] \\
        T_{\mathcal{F}'} \arrow[r, hook, "\phi'"] & TX
    \end{tikzcd}
\end{equation*}
With this definition, isomorphic distributions defines multiple sections in (\ref{quasiespdist}), therefore every isomorphism class of a codimension \(r\) distribution defines an element in
\begin{equation}\label{espaçodist}
    \mathbb{P} \big( H^0\big(X, \bigwedge^{n-r} TX \otimes \det(\Tf)^{\vee} \big) \big) = \mathbb{P} \big(H^0\big(X,\Omega^r_X \otimes \det(TX) \otimes \det(\Tf)^{\vee} \big) \big).
\end{equation}

\paragraph{}
A distribution is called involutive if satisfies \([\phi(\Tf),\phi(\Tf)] \subset \phi(\Tf)\) and in this case, by Frobenius Theorem the distribution is integrable, that is tangent to a foliation. When the distribution is integrable we will simply call it a foliation.

\begin{remark}
    A 1 dimensional distribution must be a foliation, since locally is generated by one vector field and locally it's Lie bracket is 0. We call a 1 dimensional distribution by foliation by curves.
\end{remark}

A major problem on the theory of foliations would be the classification of the irreducible components of the space of foliations, that is the Zariski open subset of (\ref{espaçodist}) with integrable sections. On the projective space some authors have made a classification for codimension one foliations such as \cite{jouanolou2006equations} for degree 0 and 1, \cite{cerveau1996irreducible} for degree 2 and \cite{da2022codimension} for a partial classification of degree 3.

\paragraph{}
In the other direction, a global section \(\omega \in H^0(X,\Omega_X^r \otimes \mathcal{L})\) with \(\mathcal{L} \in \textnormal{Pic}(X)\) defines a codimension \(r\) distribution on \(X\) if:
\begin{itemize}
    \item \(\textnormal{Codim(Sing (}\omega)) \geq 2\);

    \item For \(U \subset X\) an open set and any \(p \in U \setminus \textnormal{Sing}(\omega)\), there exists a neighborhood \(p \in V \subset U\) and 1-forms \(\eta_1, \dots, \eta_r\) on \(V\) such that
    \begin{equation*}
        \omega \mid_{V} = \eta_1 \wedge \dots \wedge \eta_r.
    \end{equation*}

    \item \(\omega\) defines a foliation if and only if 
    \begin{equation*}
        d\eta_j \wedge \eta_1 \dots \wedge \eta_r =0 , \ \ \textnormal{for all} \ \ j = 1,\dots,r.
    \end{equation*}
\end{itemize}

\paragraph{}
As we will see in the projective space, the concept of a subdistribution provides a useful framework for translating the classification of distributions of different codimensions.

\begin{definition}
    Let $\mathcal{F}$ be a codimension $r$ distribution on $X$. A subdistribution $\mathcal{G}$ of $\mathcal{F}$, is a codimension $r'\ ( \ > r)$ distribution whose tangent sheaf $\Tg$ is a subsheaf of $\Tf$.
\end{definition}

The quotient $N_{\mathcal{G}/\mathcal{F}} := \Tf/\Tg$ is called the relative normal sheaf and using the comutative diagramm
\begin{equation*}
    \begin{tikzcd}
& 0 \arrow{d} & 0 \arrow{d} \\ & \Tg \arrow{d} \arrow[equal]{r} & \Tg \arrow{d}  &   \\ 0 \arrow{r} & \Tf \arrow{d} \arrow{r} & TX \arrow{d} \arrow{r} & \Nf \arrow[equal]{d} \arrow{r} & 0 \\
 0 \arrow{r} & N_{\mathcal{G}/\mathcal{F}} \arrow{d} \arrow{r} & N_{\mathcal{G}} \arrow{r} \arrow{d} & \Nf  \arrow{r} & 0 \\ & 0 & 0 & 
\end{tikzcd}
\end{equation*}
The relative normal sheaf satisfy the exact sequence
\begin{equation*}
    0 \longrightarrow N_{\mathcal{G}/\mathcal{F}} \longrightarrow \Ng \longrightarrow \Nf \longrightarrow 0,
\end{equation*}
which shows that $N_{\mathcal{G}/\mathcal{F}}$ is a torsion-free sheaf of rank $r'-r$. If \(\mathcal{G}\) is a foliation, then is call a subfoliation.

\begin{remark}
    When \(\mathcal{F}\) is not a foliation by curves, a nonzero global section in \(H^0(X,\Tf \otimes \mathcal{L})\) with \(\mathcal{L} \in \textnormal{Pic}(X)\), defines a subfoliation by curves with tangent sheaf \(\mathcal{L}^{\vee}\).
\end{remark}

\paragraph{}
To finish this part, we give a brief introduction to the moduli space of distributions on \(X\). Fix \(P \in \mathbb{Q}[t]\) a rational polynomial and consider the functor
\begin{equation*}
    Dist^{P}_{X}(S) := 
\left\{
\begin{array}{l}
(\mathbf{F}, \phi) \mid \text{ such that } \mathbf{F} \xrightarrow{\phi} TX_{S} \text{ is a flat family of} \\
\ \ \ \ \ \ \ \ \ \ \ \ \ \  \text{distributions, with } \chi(\mathbf{F}{|_s}(t)) = P(t),  \forall \ s \in S
\end{array}
\right\} \big/ \sim.
\end{equation*}

The authors showed in \cite[Proposition 2.4.]{CCJ} that this functor is representable by a quasi-projective scheme \(D^P(X)\). Hence, we have a functorial isomorphism
\begin{equation*}
    Dist_X^P(-) \xlongrightarrow{\sim} \Hom(-, D^P(X)).
\end{equation*}
The scheme \(D^P(X)\) is defined as the open subset of \(\textnormal{Quot}^{P_{TX}-P}(TX)\) whose points parametrize torsion-free quotients of \(TX\). This construction does not impose any stability condition on the tangent sheaf associated with the family of distributions.

\subsection{Stability and Nonstability of Reflexive Sheaves}
\paragraph{}
Stability of sheaves is a fundamental concept in algebraic geometry, and stable sheaves are the building blocks of sheaf theory. However, when additional structure is imposed on a sheaf, such as being the tangent sheaf of a distribution, stability cannot be guaranteed. Slope stability requires fixing a polarization on the variety, however since we will be working with the projective space, the polarization is canonical. 
\begin{definition}\label{Stability}
    A torsion-free sheaf \(F\) on \(\pn\) is \(\mu\)-stable (resp. \(\mu\)-semistable) if, for every proper coherent subsheaf \(F' \subset F\) with \(0 < \textnormal{rank}(F') < \textnormal{rank}(F)\), we have
    \begin{equation*}
        \mu(F') < \mu(F),
    \end{equation*}
    (resp. \(\leq\)), where \(\mu = c_1/\textnormal{rank}\) is the slope of the sheaf.
\end{definition}

\begin{remark}
    In this paper, we deal with rank 2 reflexive sheaves on \(\p\). In this context, slope stability is equivalent to Gieseker stability, and since slope stability is a numerical condition, we will focus on this concept. Since we don't need to specify the type of stability, we will simply call a sheaf stable. For a complete treatment of Gieseker stability, we recommend \cite{huybrechts2010geometry}.
\end{remark}

There is a simple criterion for verifying the stability of reflexive sheaves using the normalization of a reflexive sheaf. Let $F$ be a rank 2 reflexive sheaf on $\pn$, the normalization of $F$ is defined as
\begin{equation*}
    F_{\eta} = 
    \begin{cases}
        F(-c_1(F)/2), \ \textnormal{if} \ \ c_1 \ \ \textnormal{is even,} \\
        F(-(c_1+1)/2), \ \textnormal{if} \ \ c_1 \ \ \textnormal{is odd,}
    \end{cases}
\end{equation*}
where \(F(d) = F \otimes \opn(d)\). We say that a sheaf \(F\) is normalized if \(F = F_{\eta}\). Notice that the first Chern class of the normalization is \(0\) or \(-1\). Since \(F\) is stable if and only if \(F(d)\) is stable for every \(d \in \Z\), then \(F\) is stable if and only if \(F_{\eta}\) is stable.

\begin{lemma}[\cite{RH2}]\label{stabfeixe}
    Let \(F\) be a rank 2 reflexive sheaf on \(\pn\). Then \(F\) is stable if and only if \(H^0(F_{\eta}) = 0\). If \(c_1\) is even, then \(F\) is \(\mu\)-semistable if and only if \(H^0(F_{\eta}(-1)) = 0\).
\end{lemma}

To say that a reflexive sheaf is unstable is the same as saying that it fails to satisfy Definition \ref{Stability}, in the sense that it's not even \(\mu\)-semistable. Since this is equivalent to failing the above lemma, we will use definition of a new invariant first introduced in \cite{sauer1984nonstable}.

\begin{definition}\label{Nonstability}
    Let $F$ be a rank 2 reflexive sheaf on $\pn$. Then $F$ is unstable of order $r$ if there is a positive integer $r$ maximal with the property
    \begin{equation*}
        H^0((F_{\eta})^{\vee}(-r)) \neq 0.
    \end{equation*}
\end{definition}

We call \(r\) the order of nonstability of the sheaf \(F\). As an immediate consequence, if \(r = 0\), then \(F\) is stable for \(c_1\) odd and strictly \(\mu\)-semistable for \(c_1\) even. So when we assume that a reflexive sheaf is unstable, we assume that \(r > 0\).

\paragraph{}
The maximal property of \(r\) implies that the section \(\sigma \in H^0((F_{\eta})^{\vee}(-r))\) vanishes in a pure 1-dimensional Cohen-Macaulay subscheme of \(\p\). In fact, let \(n := \textnormal{min}\{k \in \Z \mid H^0(F(k)) \neq 0 \}\) and take \(\sigma \in H^0(F(n))\) a nonzero section. Suppose that \(\textnormal{coker}(\sigma)\) has torsion with \(T(\sigma)\) it's maximal torsion subsheaf. Then we would obtain the exact sequence
\begin{equation*}
    0 \longrightarrow \op(k) \longrightarrow F(n) \longrightarrow T(\sigma)/\textnormal{coker}(\sigma) \longrightarrow 0.
\end{equation*}
But this would be an absurd, because \(\op \hookrightarrow \op(k)\) implies that \(k > 0\) and then \(H^0(F(n-k)) = 0\).
%\begin{equation*}
%    \begin{tikzcd}
%&  &  & 0 \arrow{d} \\ & 0 \arrow{d}  &   & T(\sigma) \arrow{d}   \\ 0 \arrow{r} & \op \arrow{d} \arrow{r} & F(n) \arrow[equal]{d} \arrow{r} & \textnormal{coker}(\sigma) \arrow{d} \arrow{r} & 0 \\
%0 \arrow{r} & \op(k) \arrow{d} \arrow{r} & F(n) \arrow{r}  & T(\sigma)/\textnormal{coker}(\sigma) \arrow{d}  \arrow{r} & 0 \\ & T(\sigma) \arrow{d} &  &  0 \\ & 0
%\end{tikzcd}
%\end{equation*}
%This would implie that \(k > 0\) and we would obtain an absurd, so \(\textnormal{coker}(\sigma)\) is a torsion-free rank 1 sheaf. Returning to the section \(\sigma \in H^0((F_{\eta})^{\vee}(-r))\), we would obtain an exact sequence
Returning to the original section \(\sigma \in H^0((F_{\eta})^{\vee}(-r))\), we have the exact sequence
\begin{equation*}
    0 \longrightarrow \op \xlongrightarrow{\sigma} (F_{\eta})^{\vee}(-r) \longrightarrow \mathcal{I}_Y(-r-c_1(F_{\eta})) \longrightarrow 0.
\end{equation*}
where \(Y\) is a pure 1-dimensional Cohen-Macaulay subscheme of \(\p\) defined by the zeros of \(\sigma\). We call \(Y\) the curve associated to the sheaf \(F\). Notice that, assuming \(r > 0\) the space \(H^0((F_{\eta})^{\vee}(-r))\) is 1-dimensional.

\section{Distributions and Foliations on \(\p\)}\label{section 3}
\paragraph{}
Here we will extend the definitions of the last section to the case \(X = \p\) and discuss some of its principal properties and invariants. Our main goal is to study codimension one distributions using the minimal possible degree of the subfoliation by curves that they admit.

\subsection{Foliation by Curves on \(\p\)}
\paragraph{}

A foliation by curves \(\mathcal{G}\) of degree \(d'\) in \(\p\) is given by
\begin{equation*}
    0 \longrightarrow \op(1-d') \longrightarrow \tpb \longrightarrow \Ng \longrightarrow 0,
\end{equation*}
where \(\Ng\) is a rank 2 torsion-free sheaf. Taking the dual sequence, we obtain
\begin{equation}\label{foliporcurvas}
    0 \longrightarrow \Ng^{\vee} \longrightarrow \Omega^1_{\mathbb{P}^3} \longrightarrow \mathcal{I}_W(d'-1) \longrightarrow 0
\end{equation}
where \(\Ng^{\vee}\) is a rank 2 reflexive sheaf and \(W = \textnormal{Sing}(\mathcal{G})\). The sheaf \(\Ng^{\vee}\) is called the conormal sheaf of \(\mathcal{G}\). Now consider $\mathcal{U}'$ the maximal zero-dimensional subsheaf of $\mathcal{O}_W$, so that the quotient sheaf is the structure sheaf of a subscheme $C' \subset W \subset \p$ of pure dimension 1. We call 
\begin{center}
    \(C' = \textnormal{Sing}_1(\mathcal{G})\) the 1-dimensional component of \(\textnormal{Sing}(\mathcal{G})\).
\end{center}
\begin{center}
    \(\textnormal{Supp}(\mathcal{U}') = \textnormal{Sing}_0(\mathcal{G})\) the 0-dimensional component of \(\textnormal{Sing}(\mathcal{G})\).
\end{center}

\begin{remark}
    From another point of view, the foliation by curves \(\mathcal{G}\) defines a section \(v \in H^0(\tpb(d'-1))\). Consider the twisted Euler sequence
    \begin{equation*}
        0 \longrightarrow \op(d' - 1) \longrightarrow \op(d')^{\oplus \ 4} \longrightarrow \tpb(d'-1) \longrightarrow 0,
    \end{equation*}
    in homogeneous coordinates \(x_i\), we may represent the global section as 
    \begin{equation*}
        v = \sum F_i \frac{\partial}{\partial x_i} \ \ \textnormal{with} \ \ F_i \in H^0(\op(d')).
    \end{equation*}
    Two representations \(\{F_i\}\) and \(\{G_i\}\) coincide if there exists \(f \in H^0(\op(d'-1))\) such that
    \begin{equation*}
        F_i - G_i =  x_if, \ \ \textnormal{for every} \ i=0,\dots,3.
    \end{equation*}
    The singular scheme \(\textnormal{Sing}(\mathcal{G})\) is defined by the homogeneous ideal generated by the \(2 \times 2\) minors of the matrix
    \begin{equation*}
        \begin{pmatrix}
            F_0 & F_1 & F_2 & F_3 \\
            x_0 & x_1 & x_2 & x_3
        \end{pmatrix}.
    \end{equation*}
\end{remark}

Using the exact sequence (\ref{foliporcurvas}) we may correlate the Chern classes of the conormal sheaf with the invariants of the singular scheme.
\begin{theorem}[\cite{correa2023classification}]\label{invariantesfolporcurv}
    Let $\mathcal{G}$ be a foliation by curves of degree $d'$ in $\p$ and consider the construction above for $C'$ and $\mathcal{U}'$. Then,
    \begin{align*}
        c_1(\Ng^{\vee}) & = -3-d' \\
        c_2(\Ng^{\vee}) &= d'^2 +2d'+3 - \textnormal{deg}(C') \\
        c_3(\Ng^{\vee}) &= \textnormal{length}(\mathcal{U}) = d'^3 + d'^2 + d' - 3\textnormal{deg}(C')(d'-1) + 2p_a(C) - 1
    \end{align*}
\end{theorem}

In \cite[Theorem 4.]{galeano2022codimension}, the authors classify the conormal sheaf \(\Ng^{\vee}\) of foliation by curves of degree 1. Their result was that 
\begin{enumerate}
    \item \(\Ng^{\vee}\) is stable with \(c_2 = 6\), \(c_3 = 4\) and \(\textnormal{Sing}(\mathcal{G})\) is a 0-dimensional scheme of length 4; or

    \item \(\Ng^{\vee}\) is strictly \(\mu\)-semistable with \(c_2 = 5\), \(c_3 = 2\) and \(\textnormal{Sing}(\mathcal{G})\) is the union of a line with a 0-dimensional scheme of length 2; or

    \item \(\Ng^{\vee} = \op(-2)^{\oplus 2}\) and \(\textnormal{Sing}(\mathcal{G})\) is a pair of skew lines (or a double line of genus -1).
\end{enumerate}

\paragraph{}
Let's investigate the foliation by curves whose 1-dimensional component of the singular scheme is a line, along with some isolated points. Let \(\mathcal{G}\) be a foliation by curves of degree \(d'\) such that \(\textnormal{Sing}_1(\mathcal{G}) = L\) is a line, then
\begin{align*}
    c_1(\Ng^{\vee}) &= -3 -d' \\
    c_2(\Ng^{\vee}) &= d'^2 + 2d' + 2 \\
    c_3(\Ng^{\vee}) &= d'^3 + d'^2 -2d' + 2 
\end{align*}

Such a foliation is defined by an element of \(H^0(\tpb(d'-1) \otimes \mathcal{I}_L)\). Using the exact sequence of the line as a complete intersection we obtain
\begin{equation*}
    0 \longrightarrow \tpb(d'-3) \longrightarrow \tpb(d'-2)^{\oplus 2} \longrightarrow \tpb(d'-1) \otimes \mathcal{I}_L \longrightarrow 0.
\end{equation*}
Since \(\tpb(d'-2)\) is globally generated for \(d' \geq 1\), so is \(\tpb(d'-1) \otimes \mathcal{I}_L\). This means that the space \(H^0(\tpb(d'-1) \otimes \mathcal{I}_L)\) has a genereic section for each \(d' \geq 1\). This guarantee the existence of foliation by curves vanishing on the line \(L\).

\subsection{Codimension One Distributions on $\p$}

\paragraph{}
A codimension one distribution $\mathcal{F}$ of degree $d$, is given by an exact sequence
\begin{equation}\label{eq1}
    0 \longrightarrow \Tf \longrightarrow \tpb \longrightarrow \mathcal{I}_Z(d+2) \longrightarrow 0,
\end{equation}
with $\Tf$ a rank 2 reflexive sheaf and $Z = \textnormal{Sing}(\mathcal{F})$ the singular scheme of $\mathcal{F}$. As the same process of the foliation by curves, we might consider $\mathcal{U}$ the maximal zero-dimensional subsheaf of $\mathcal{O}_Z$, so that the quotient sheaf is the structure sheaf of a subscheme $C \subset Z \subset \p$ of pure dimension 1. We call 
\begin{center}
    \(C = \textnormal{Sing}_1(\mathcal{F})\) the 1-dimensional component of \(\textnormal{Sing}(\mathcal{F})\).
\end{center}
\begin{center}
    \(\textnormal{Supp}(\mathcal{U}) = \textnormal{Sing}_0(\mathcal{F})\) the 0-dimensional component of \(\textnormal{Sing}(\mathcal{F})\).
\end{center}

\begin{remark}
    As we have seen, the codimension one distribution \(\mathcal{F}\) defines a section \(\omega \in H^0(\Omega_{\p}^1(d+2))\). The twisted Euler sequence
    \begin{equation*}
        0 \longrightarrow \Omega_{\p}^1(d+2) \longrightarrow \op(d+1)^{\oplus\ 4} \longrightarrow \op(d+2) \longrightarrow 0,
    \end{equation*}
    shows that in homogeneous coordinates \(x_i\), we may write the section as 
    \begin{equation*}
        \omega = \sum_{i=0}^3 A_i dx_i \  \ \textnormal{with} \ \ A_i \in H^0(\p,\op(d+1)),
    \end{equation*}
    where the homogeneous polynomials \(A_i\) must satisfy the condition
    \begin{equation*}
        \sum_{i=0}^3 A_i x_i = 0.
    \end{equation*}
\end{remark}

\begin{theorem}[\cite{CCJ}]\label{invariantes}
    Let $\mathcal{F}$ be a codimension one distribution of degree $d$ in $\p$ and consider the construction above for $C$ and $\mathcal{U}$. Then,
    \begin{align*}
        c_1(\Tf) & = 2-d \\
        c_2(\Tf) &= d^2 +2 - \textnormal{deg}(C) \\
        c_3(\Tf) &= \textnormal{length}(\mathcal{U}) = d^3 + 2d^2 + 2d - \textnormal{deg}(C)(3d-2) + 2p_a(C) - 2
    \end{align*}
\end{theorem}
Notice that the singular scheme \(\textnormal{Sing}(\mathcal{F})\) has no isolated points if and only if \(\Tf\) is locally free, this is a particular case of a more general result \cite[Lemma 2.1]{CCJ}. 

\begin{example}\label{nullcorrelation}
    The Null-correlation bundle \(N\) defines a natural structure of a codimension one distribution on \(\p\). This bundle is defined as the kernel of an epimorphism showed by the exact sequence
    \begin{equation*}
        0 \longrightarrow N \longrightarrow \tpb(-1) \longrightarrow \op(1) \longrightarrow 0.
    \end{equation*}
    If we tensor by \(\op(1)\) we have
    \begin{equation*}
        0 \longrightarrow N(1) \longrightarrow \tpb \longrightarrow \op(2) \longrightarrow 0,
    \end{equation*}
    which is a regular codimension one distribution of degree 0. In fact, this is the only case of a regular distribution on \(\p\), as showed in \cite{CCJ}. Since foliations on the projective space always have singularities, this is an example of a non-integrable distribution on \(\p\).
\end{example}

%By \cite[Corollary 4.8]{soares2005holomorphic}, we obtain upper bounds 
%\begin{equation*}
%    \textnormal{deg}(C) \leq d^2 + d + 1
%\end{equation*}
%and therefore lower bounds for $c_2(\Tf)$. Combining the results we get
%\begin{equation}\label{c2tf}
%    1 - d \leq c_2(\Tf) \leq d^2 + 2.
%\end{equation}

\paragraph{}
Fixing a rank 2 reflexive sheaf on $\p$ and constructing an explicit map to the tangent sheaf of $\p$ that is injective and has torsion-free cokernel can be a very hard task. However, in \cite[Appendix 1]{CCJ} the authors showed that if \(G\) is a rank 2 reflexive sheaf globally generated, then \(G^{\vee}(1)\) is the tangent sheaf of a codimension one distribution of degree \(c_1(G)\). We may use this result to give some cohomological conditions that guarantee the existence of distributions.

\begin{corollary}\label{regular}
    Let $d \geq 0$ and fix $F$ a rank 2 reflexive sheaf on $\p$ with $c_1(F) = 2-d$. If $F$ is a $(d-1)$-regular sheaf, then there exist a codimension one distribution $\mathcal{F}$ of degree $d$ with tangent sheaf $\Tf \cong F$.
\end{corollary}
\begin{proof}
    Since $F$ is $(d-1)$-regular, by a Castelnuovo-Mumford Theorem \cite[Page 99]{mumford2016lectures}, we know that $F(d-1)$ is globally generated. Then $F(d-1)^{\vee}(1)$ is the tangent sheaf of a codimension one distribution of degree $c_1(F(d-1)) = c_1(F) + 2 (d-1) = d$. Also,
    \begin{equation*}
        F(d-1)^{\vee}(1) \cong F^{\vee}(2-d) \cong F(-c_1 + 2- d) \cong F.
    \end{equation*}
    With this we obtain the result.
\end{proof}
In summary, fix \(F\) a rank 2 reflexive sheaf with \(c_1(F) = 2-d\) and \(d \geq 0\). If
\begin{equation*}
    H^i(F(d-1-i)) = 0 \ \ \textnormal{for} \ \ i = 1,2,3,
\end{equation*}
then there exists a codimension one distribution of degree \(d\) with tangent sheaf \(F\).

\paragraph{} 
Now, let us define the following invariant of a codimension one distribution \(\mathcal{F}\) on \(\p\):
\begin{definition}\label{menorseçao}
    Let \(\mathcal{F}\) be a codimension one distribution on \(\p\). We define
    \begin{equation*}
    t_{\mathcal{F}} := \textnormal{min}\{k \in \Z \mid H^0(\Tf(k -1)) \neq 0 \},
\end{equation*}
the minimum integer such that the tangent sheaf of \(\mathcal{F}\) has a global section. Since \(\Tf\) is a subsheaf of \(\tpb\), we get that \(t_{\mathcal{F}} \geq 0\).
\end{definition}

Our main interest in this paper is to study codimension one distributions with fixed \(t_{\mathcal{F}}\). As we will see in the next section, for small values of \(t_{\mathcal{F}}\) the sheaf \(\Tf\) must be unstable.

Let \(\sigma \in H^0(\Tf(t_{\mathcal{F}}-1))\) be a nonzero section. The minimality implies that the zeros of \(\sigma\) must be a Cohen-Macaulay curve \(Y\), and \(Y\) is nonempty if and only if \(\Tf\) is nonsplit. We can construct a subfoliation by curves $\mathcal{G}$ of degree $t_{\mathcal{F}}$ by the following commutative diagram
\begin{equation}\label{subfol}
        \begin{tikzcd}
            & 0 \arrow{d} & 0 \arrow{d} \\ & \op(1 - t_{\mathcal{F}}) \arrow{d} \arrow[equal]{r} & \op(1 - t_{\mathcal{F}}) \arrow{d}  &   \\ 0 \arrow{r} & \Tf \arrow{d} \arrow{r} & \tpb \arrow{d} \arrow{r} & \mathcal{I}_Z(d+2) \arrow[equal]{d} \arrow{r} & 0 \\
            0 \arrow{r} & \mathcal{I}_Y(1 +t_{\mathcal{F}} -d) \arrow{d} \arrow{r} & N_{\mathcal{G}} \arrow{r} \arrow{d} & \mathcal{I}_Z(d+2) \arrow{d} \arrow{r} & 0 \\ & 0 & 0 & 0
        \end{tikzcd}
\end{equation}
The subfoliation by curves \(\mathcal{G}\) is defined by the middle column. So \(t_{\mathcal{F}}\) is the smallest degree of a subfoliation by curves that \(\mathcal{F}\) admits. This construction satisfy the conditions of \cite[Proposition 3.4.]{correa2026holomorphicfoliationsdegreearbitrary} which showed that
\begin{equation}\label{Y em sing}
     Y  \subset \textnormal{Sing}(\mathcal{G}).
\end{equation}

The authors of \cite[Proposition 10.3.]{mendson2024codimension} showed a sharp bound on the minimal degree of a subfoliation by curves of \(\mathcal{F}\), by taking it's intersection with a general codimension one foliation of degree 0. In our notation is stated as
\begin{equation}\label{cotatf}
    t_{\mathcal{F}} \leq d + 1.
\end{equation}
Moreover, if \(\mathcal{F}\) is a foliation, then \(t_{\mathcal{F}} \leq d\).

\begin{remark}
    In \cite{hartshorne1982stable}, the author showed a bound on the minimal twist for which a rank 2 reflexive sheaf admits a global section in terms of the Chern classes. Using this Theorem together with Theorem \ref{invariantes}, we obtain that \(t_{\mathcal{F}} \leq 2d - 1\). This bound depends only on the sheaf, so being a subsheaf of \(\tpb\) implies a better bound as (\ref{cotatf}).
\end{remark}

\begin{proposition}\label{novaprop}
    Let \(\mathcal{F}\) be a codimension one distribution and \(\mathcal{G}\) a subfoliation by curves defined by some \(\sigma \in H^0(\Tf(t_{\mathcal{F}} - 1))\). If \(\mathcal{G}\) has only isolated singularities, then \(\Tf\) splits.
\end{proposition}
\begin{proof}
    If \(\mathcal{G}\) is a foliation by curves with only isolated singularities, then \(\textnormal{Sing}(\mathcal{G})\) is a 0-dimensional subscheme. So (\ref{Y em sing}) implies that \(Y = \emptyset\) and then \(\Tf\) splits. 
    %Now suppose \(\Tf\) splits, then \(Y = \emptyset\) and \(Z = \textnormal{Sing}(\mathcal{F})\) is a curve with no isolated points. Therefore \(\Ng\) would be defined by an element of \(\Ext^1(\mathcal{I}_Z(2d+1 - t_{\mathcal{F}}), \op)\) where \(Z\) is an ACM curve by \cite[Theorem 1]{correa2015singular}. Such extensions are in correspondence with reflexive sheaves by the Serre correspondence, so \(\Ng\) is reflexive and then \(\textnormal{Sing}(\mathcal{G}) = \textnormal{Sing}(\Ng)\) has codimension 3 on \(\p\) by \cite[Lemma 1.1.10.]{okonek1980vector}.
\end{proof} 

This shows that, for a codimension one distribution \(\mathcal{F}\) to admit a subfoliation by curves, some conditions are imposed on \(\mathcal{F}\). In \cite[Lemma 4.3]{CCJ}, the authors first presented this construction and proved that, if a codimension one distribution admits a subfoliation by curves of degree \(0\), then its tangent sheaf splits. In other words, if \(t_{\mathcal{F}} = 0\), then \(\Tf\) splits. We may rewrite this result to obtain the tangent sheaf explicitly.

\begin{corollary}\label{tf=-1}
    Let $\mathcal{F}$ be a codimension one distribution of degree $d$ on $\p$, such that \(t_{\mathcal{F}} = 0\). Then
    \begin{equation*}
        \Tf = \op(1) \oplus \op(1-d) \ \ \ \textnormal{and} \ \ d \geq 0  .
    \end{equation*}
\end{corollary}
\begin{proof}
    Since \(t_{\mathcal{F}} = 0\), \(\Tf\) must split. Since \(H^0(\Tf(-1)) \neq 0\) and \(H^0(\Tf(-2)) = 0\), \(\op(1)\) must be one of the summands, and thus the other one must be \(\op(1-d)\).
\end{proof}

\begin{remark}
    We may describe the tangent sheaf of split codimension one distributions in therms of the degree and \(t_{\mathcal{F}}\):
    \begin{equation*}
        \Tf = \op(1-t_{\mathcal{F}}) \oplus \op(1 + t_{\mathcal{F}} - d) \ \ \ \textnormal{with} \ \ \ d \geq 2t_{\mathcal{F}}
    \end{equation*}
    So in this case, the tangent sheaf is uniquely determined by those invariants. Expanding for small values we have:
    \begin{center}
    %Split Distributions in therms of \(d\) and \(t_{\mathcal{F}}\)}
    \vspace{0.1cm}
    \begin{tabular}{| c | c  c  c  c |}
        \hline
         \ \ \ $d$ \ / \  $t_{\mathcal{F}}$ \ \ \ & \ \ \ $0$ \ \ \ & \ \ \ $1$ \ \ \ & \ \ \ $2$ \ \ \ & \ \ \ $3$ \ \ \ \\
        \hline
        0 & \(\op(1) \oplus \op(1)\) & \(\times\) & \(\times\) & \(\times\)  \\
        1 & \(\op(1) \oplus \op\) & \(\times\) & \(\times\) & \(\times\)   \\
        2 & \(\op(1) \oplus \op(-1)\) & \(\op \oplus \op\) & \(\times\) & \(\times\)   \\
        3 & \(\op(1) \oplus \op(-2)\) & \(\op \oplus \op(-1)\) & \(\times\) & \(\times\)   \\
        4 & \(\op(1) \oplus \op(-3)\) & \(\op \oplus \op(-2)\) & \(\op(-1) \oplus \op(-1)\) & \(\times\)  \\
        5 & \(\op(1) \oplus \op(-4)\) & \(\op \oplus \op(-3)\) & \(\op(-1) \oplus \op(-2)\) & \(\times\)  \\
        6 & \(\op(1) \oplus \op(-5)\) & \(\op \oplus \op(-4)\) & \(\op(-1) \oplus \op(-3)\) & \(\op(-2) \oplus \op(-2)\)  \\
        \hline
    \end{tabular}
    \end{center}
    \begin{center}
        Table 1: Split Distributions in therms of \(d\) and \(t_{\mathcal{F}}\)
    \end{center}
\end{remark}

\begin{example}\label{splitruim}
    The only codimension one distributions \(\mathcal{F}\) on \(\p\) with tangent sheaf split and \(\mu\)-semistable are the first entries of every collum on Table 1. Notice that in this cases we have the equality 
    \begin{equation*}
        d = 2t_{\mathcal{F}}
    \end{equation*}
    and 
    \begin{equation*}
        \Tf= \op(1 - t_{\mathcal{F}}) \oplus \op(1-t_{\mathcal{F}}).
    \end{equation*}
    So it's easier to write every invariant of the distribution in therms of \(t_{\mathcal{F}}\). Since \(\Tf\) is locally free, the singular scheme is a pure 1-dimensional subscheme of \(\p\), that is \(\textnormal{Sing}(\mathcal{F}) = \textnormal{Sing}_1(\mathcal{F}) = C\). Using Theorem \ref{invariantes}, we obtain that
    \begin{align*}
        \deg(C) &= 3t_{\mathcal{F}}^2 + 2t_{\mathcal{F}} +1 \\
        p_a(C) &= t_{\mathcal{F}}(5 t_{\mathcal{F}}^2 - t_{\mathcal{F}} - 1)
    \end{align*}
    for every \(t_{\mathcal{F}} \geq 0\). All such distributions have the same normalization \((\Tf)_{\eta} = \op \oplus \op\).
\end{example}

Since we are interested in unstable sheaves, we will add the assumption of the tangent sheaf being nonsplit to avoid the above case.

\begin{lemma}\label{h^0=1}
    Let \(\mathcal{F}\) be a codimension one distribution of degree \(d\) on \(\p\) such that \(d > 2 t_{\mathcal{F}}\). Then \(h^0(\Tf(t_{\mathcal{F}}-1)) = 1\). Moreover, if \(\Tf\) is nonsplit the result is valid for \(d = 2t_{\mathcal{F}}\).
\end{lemma}
\begin{proof}
    Suppose that \(d > 2 t_{\mathcal{F}}\) and consider the exact sequence
    \begin{equation*}
        0 \longrightarrow \op \longrightarrow \Tf(t_{\mathcal{F}} - 1) \longrightarrow \mathcal{I}_Y(2t_{\mathcal{F} } - d) \longrightarrow 0.
    \end{equation*}
    By our assumption the sheaf \(\mathcal{I}_Y(2t_{\mathcal{F}} - d)\) cannot have a global section, even if \(Y = \emptyset\). Therefore \(h^0(\Tf(t_{\mathcal{F}}-1)) = h^0(\op) = 1\). If \(\Tf\) is nonsplit, then \(Y\) is nonempty and therefore \(H^0(\mathcal{I}_Y) = 0\).
\end{proof}

The above result gives conditions to when the minimal subfoliation by curves is unique. Notice that the Examples \ref{nullcorrelation} and \ref{splitruim} does not satisfy the relation \(d > 2 t_{\mathcal{F}}\) and both have \(h^0(\Tf(t_{\mathcal{F}} -1)) > 1\).

\begin{proposition}\label{Tfcomseção}
    Let $\mathcal{F}$ be a codimension one distribution of degree $d$, such that \(t_{\mathcal{F}} = 1\). Then
    \begin{enumerate}
        \item $\Tf = \op \oplus \op(2-d)$ with $d \geq 2$; or

        \item $\Tf$ has Chern classes $(2-d,1,d)$ with $d \geq 1$; or 

        \item $\Tf$ has Chern classes $(2-d,2,2d)$ with $d \geq 0$.
    \end{enumerate}
\end{proposition}
\begin{proof}
    Let $\mathcal{F}$ be a codimension one distribution of degree $d$ and $Z = \textnormal{Sing}(\mathcal{F})$, given by
    \begin{equation*}
        0 \longrightarrow \Tf \longrightarrow \tpb \longrightarrow \mathcal{I}_Z(d+2) \longrightarrow 0.
    \end{equation*}
    Take $\sigma \in H^0(\Tf)$ a nonzero section. Since is the first global section by a twist, we know that the zeros $Y :=(\sigma)_0$ will be a codimension 2 subscheme of $\p$. If $Y$ is empty, then $\Tf$ is  split and by our hypothesis one of the copies must be $\op$, so we obtain
    \begin{equation*}
        \Tf = \op \oplus \op(2-d) \ \ \textnormal{and} \ \ d \geq 2.
    \end{equation*}
    
    In the following, we suppose that $Y$ is nonempty. The construction on (\ref{subfol}) show us that the composition $ \op \xhookrightarrow{\sigma} \Tf \hookrightarrow \tpb$ defines a subfoliation of degree 1, denoted by $\mathcal{G}$, and the comutative diagram
    \begin{equation*}
        \begin{tikzcd}
            & 0 \arrow{d} & 0 \arrow{d} \\ & \op \arrow{d} \arrow[equal]{r} & \op \arrow{d}  &   \\ 0 \arrow{r} & \Tf \arrow{d} \arrow{r} & \tpb \arrow{d} \arrow{r} & \mathcal{I}_Z(d+2) \arrow[equal]{d} \arrow{r} & 0 \\
            0 \arrow{r} & \mathcal{I}_Y(2-d) \arrow{d} \arrow{r} & N_{\mathcal{G}} \arrow{r} \arrow{d} & \mathcal{I}_Z(d+2) \arrow{d} \arrow{r} & 0 \\ & 0 & 0 & 0
        \end{tikzcd}
    \end{equation*}
    Dualizing the bottom line we get
    \begin{equation*}
        0 \rightarrow \op(-d-2) \rightarrow \Ng^{\vee} \rightarrow \op(d-2) \xrightarrow{\xi} \omega_C(2-d) \rightarrow
    \end{equation*}
    \begin{equation*}
        \rightarrow \ext^1(\Ng,\op) \rightarrow \omega_Y(2+d) \rightarrow \ext^3(\mathcal{U},\op) \rightarrow 0,
    \end{equation*}
    where $C = \textnormal{Sing}_1(\mathcal{F})$ and $\textnormal{Supp}(\mathcal{U}) = \textnormal{Sing}_0(\mathcal{F})$. We may break the exact sequence by defining the subscheme $K$ of \(\p\) as $\textnormal{ker}(\xi) = \mathcal{I}_K(d-2)$ and so
    \begin{equation*}
        \begin{tikzcd}
            &  &  & \mathcal{I}_K(d-2) \arrow[hookrightarrow]{d}  \\ 0 \arrow{r} & \op(-d-2)  \arrow{r} & \Ng^{\vee} \arrow[twoheadrightarrow]{ur} \arrow{r} & \op(d-2) \arrow[twoheadrightarrow]{d}\arrow{r} & \omega_C(2-d) \arrow{r} & \dots \\
            &    &   & \mathcal{O}_K(d-2) \arrow[hookrightarrow]{ur}    &
        \end{tikzcd}
    \end{equation*}
    So we get the short exact sequence
    \begin{equation*}
        0 \longrightarrow \op (-d-2) \longrightarrow \Ng^{\vee} \longrightarrow \mathcal{I}_K(d-2) \longrightarrow 0,
    \end{equation*}
    and then $K$ is the zero of a general section in $H^0(\Ng^{\vee}(d+2))$. Therefore
    \begin{equation*}
        \deg(K) = c_2(\Ng^{\vee}(d+2)) = c_2(\Ng^{\vee}) + d^2 - 4.
    \end{equation*}

    By the factorization, we have the inclusion $\mathcal{O}_K(d-2) \hookrightarrow \omega_C(2-d)$ which gives inclusion on the support 
    \begin{equation*}
        K \subseteq C.
    \end{equation*} 
    Since both schemes are pure of dimension 1, we gain the inequality
    \begin{equation*}
        \deg(K) \leq \deg(C).
    \end{equation*}
    Now, $\deg(Y) = c_2(\Tf)$ and $Y$ is nonempty, so $c_2(\Tf) > 0$ and by Theorem \ref{invariantes}, 
    \begin{equation*}
        \deg(C) < d^2 + 2.
    \end{equation*}
    
    Using the classification of degree 1 foliation by curves in \cite[Theorem 4.]{galeano2022codimension}, we may calculate every possibility of $\deg(K)$ and compare with the above condition on $C$.
    \begin{enumerate}
        \item If $\Ng^{\vee}$ is $\mu$-stable with $c_2(\Ng^{\vee}) = 6$, then it has only isolated singularities and by Proposition \ref{novaprop} \(\Tf\) would be split, which is an absurd.

        \item If $\Ng^{\vee}$ is strictly $\mu$-semistable with $c_2(\Ng^{\vee}) = 5$, then \(\deg(K) = d^2 + 1\), which implies that $\deg(C) = d^2 + 1$ and therefore \(K = C\). So we have 
        \begin{equation*}
            \deg(Y) = c_2(\Tf) = d^2 + 2 - \deg(C) = 1,
        \end{equation*}
        and $Y = L$ is a line. Since $Y$ is a line, then $p_a(Y) = 0$ and then
        \begin{equation*}
            c_3(\Tf) = 2p_a(Y) - 2 + 1(4-c_1(\Tf)) = -2+d+2 = d.
        \end{equation*}
        So the Chern classes of \(\Tf\) are $(c_1,c_2,c_3) = (2-d,1,d)$. In the case \(d = 0\), 
        the exact sequence
        \begin{equation*}
            0 \longrightarrow \op \longrightarrow \Tf \longrightarrow \mathcal{I}_L(2) \longrightarrow 0,
        \end{equation*}
        is a twist of the sequence of a line, which implies that \(\Tf = \op(1) \oplus \op(1)\), but this means that \(t_{\mathcal{F}} = 0\). Therefore, \(d \geq 1\).
        
        \item If $\Ng^{\vee} = \op(-2)^{\oplus 2}$, then $K$ is a complete intersection of degree $d^2$. We have two possibilities for $C$, one is $\deg(C) = d^2+1$, which is the same as the above case, so we may assume that $\deg(C) = d^2$ and then $C = K$. Therefore,
        \begin{equation*}
            \deg(Y) = c_2(\Tf) = d^2+2-\deg(C) = 2.
        \end{equation*}
        Since $C$ is the complete intersection of two hypersurfaces of degree $d$,
        \begin{equation*}
            p_a(C) = \frac{1}{2}dd(d+d-4)+1 = d^2(d-2)+1 = d^3 -2d^2 +1.
        \end{equation*}
        Applying the result on Theorem \ref{invariantes},
        \begin{align*}
            c_3(\Tf) &= d^3 + 2d^2 + 2d - \deg(C)(3d-2) + 2p_a(C) -2 \\
            &= d^3 + 2d^2 + 2d - d^2(3d-2) + 2(d^3 -2d^2 +1) -2 \\
            &= 2d.
        \end{align*}
        We conclude that the invariants of $\Tf$ will be $(2-d,2,2d)$. In this last case, for \(d = 0\) we have that \(\Tf\) have invariants \((2,2,0)\) and must be \(\Tf = N(1)\), as in Example \ref{nullcorrelation}. So \(t_{\mathcal{F}} = 1\) for every \(d \geq 0\).
    \end{enumerate}
\end{proof}

Notice that in the cases where the tangent sheaf is locally free and admits a subfoliation of degree 0 or~1, it either splits or is the null-correlation bundle in Example~\ref{nullcorrelation}, which is not a foliation. In other words, we have shown that if a codimension one foliation on \(\p\) with locally free tangent sheaf admits a subfoliation of degree 0 or 1, then \(\Tf\) splits. Our result agrees with \cite[Theorem~C]{dacosta2024splittingaspectsholomorphicdistributions}, which independently first proved this fact.

\begin{lemma}
    Let \(\mathcal{F}\) be a codimension one distribution on \(\p\) such that \(\Tf\) is split and \(\mathcal{G}\) the subfoliation by curves induced by \(\sigma \in H^0(\Tf(t_{\mathcal{F}}-1))\). Then
    \begin{equation*}
        \textnormal{Sing}(\mathcal{G}) \subset \textnormal{Sing}(\mathcal{F}).
    \end{equation*}
\end{lemma}
\begin{proof}
    The exact sequence of the relative normal sheaf would be 
    \begin{equation*}
        0 \longrightarrow \op(1+t_{\mathcal{F}}-d) \to \Ng \longrightarrow \mathcal{I}_Z(d+2) \longrightarrow 0,
    \end{equation*}
    Where \(Z = \textnormal{Sing}(\mathcal{F})\). Dualizing this sequence we have
    \begin{equation*}
        0 \longrightarrow \op(-2-d) \longrightarrow \Ng^{\vee} \longrightarrow \op(d - 1 - t_{\mathcal{F}}) \longrightarrow \omega_Z(d-2) \longrightarrow \ext^1(\Ng,\op) \longrightarrow 0.
    \end{equation*}
    Since the last map is surjective, then every point on \(\textnormal{Supp}(\ext^1(\Ng,\op))\) is on \(Z\). Therefore we have
    \begin{equation*}
        \textnormal{Sing}(\mathcal{G}) = \textnormal{Supp}(\ext^1(\Ng,\op)) \subset \textnormal{Sing}(\mathcal{F}).
    \end{equation*}
\end{proof}

Now let's study the codimension one distributions that contain the foliation by curves vanishing on a line with isolated points. Let \(\mathcal{F}\) be a codimension one distribution on \(\p\) such that the non-trivial section \(\sigma \in H^0(\Tf(t_{\mathcal{F}}-1))\) defines a foliation by curves \(\mathcal{G}\) with \(\textnormal{Sing}_1(\mathcal{G}) = L\) a line. Assuming \(\Tf\) is nonsplit, the zeros of the section \(\sigma\) is non-empty. But \(Y \subseteq L\) by (\ref{Y em sing}), so we obtain \(Y = L\). This means that \(c_2(\Tf(t_{\mathcal{F}}-1)) = 1\) and then
\begin{equation*}
    \begin{cases}
        c_1(\Tf) = 2-d \\
        c_2(\Tf) = -t_{\mathcal{F}}^2 + d(t_{\mathcal{F}}-1) +2 \\
        c_3(\Tf) = d - 2(t_{\mathcal{F}}-1)
    \end{cases}
\end{equation*}
with \(d \geq 2(t_{\mathcal{F}}-1)\). If one wants to classify codimension one distributions by fixing \(t_{\mathcal{F}}\), the above case will always be among the classification for every \(t_{\mathcal{F}} \in \Z_{\geq0}\).

%\paragraph{}
%We may repeat this argument for an arbitrary \(t_{\mathcal{F}}\) and obtain that
%\begin{equation*}
%    0 \leq c_2(\Tf(t_{\mathcal{F}}-1)) \leq \deg(C')
%\end{equation*}
%\begin{equation*}
%    0 \leq c_2(\Ng^{\vee}(d+2)) \leq \deg(C)
%\end{equation*}

\section{Codimension One Distributions with Unstable Tangent Sheaf}\label{section 4}
\paragraph{}
Unstable reflexive sheaves were studied in \cite{sauer1984nonstable}, where the author showed that properties initially considered only for stable sheaves, such as vanishing conditions on cohomology groups and the spectrum, could be extended to unstable sheaves using an invariant called the order of nonstability. Now let's apply Definition \ref{Nonstability} to the tangent sheaf of a codimension one distribution. Let $\mathcal{F}$ be a codimension one distribution of degree $d$ on $\p$ and define \(\varepsilon\) as \(\varepsilon \in \{0,1\}\) and \(\varepsilon \equiv d \ (\textnormal{mod }2)\) . Since $c_1(\Tf) = 2-d$, its normalization is
\begin{equation*}
    (\Tf)_{\eta} = \Tf\big(\frac{d-2- \varepsilon}{2}\big).
\end{equation*}
Now suppose $\Tf$ is unstable of order $r$, then 
\begin{equation*}
    H^0((\Tf)_{\eta}^{\vee}(-r)) \neq 0.
\end{equation*}
Expanding we have that
\begin{equation*}
    H^0((\Tf)_{\eta}^{\vee}(-r)) = H^0 \big(\Tf\big(\frac{d-2 + \varepsilon}{2} - r\big)\big)
\end{equation*}
By the minimality property on \(r\) and Definition \ref{menorseçao}, it follows
\begin{equation*}
    t_{\mathcal{F}} = \frac{d+\varepsilon}{2} - r.
\end{equation*}
Notice that the order of nonstability and \(t_{\mathcal{F}}\) are inversely related, so the more unstable \(\Tf\) is, the lower is the degree of the minimal subfoliation by curves. Since both are positive integers, then \(r \geq 1\) and with Lemma \ref{stabfeixe} we obtain the following result.

\begin{lemma}\label{lemmaintro}
    Let \(\mathcal{F}\) be a codimension one distribution of degree \(d\) on \(\p\) such that \(\Tf\) is nonsplit. Then \(\Tf\) is unstable if and only if \(d \geq 3\) and
    \begin{equation*}
        1 \leq t_{\mathcal{F}} \leq \frac{d-2+ \varepsilon}{2},
    \end{equation*}
    where \(\varepsilon \in \{0,1\}\) and \(\varepsilon \equiv d (\textnormal{mod }2)\). Moreover, \(\Tf\) will be unstable of order \(\frac{d+ \varepsilon}{2} - t_{\mathcal{F}}\).
\end{lemma}

\paragraph{}
The above lemma shows that when \(\Tf\) is nonsplit and unstable, it satisfies the conditions of Lemma \ref{h^0=1}. In other words, if \(\mathcal{F}\) is a codimension one distribution with \(\Tf\) nonsplit and unstable, then the minimal subfoliation by curves is unique.

When \(t_{\mathcal{F}}\) attains its minimum value, \(t_{\mathcal{F}} = 1\), this is exactly when the order of nonstability attains its maximum possible value. The codimension one distributions with this property are classified in Proposition \ref{Tfcomseção}, so we obtain the next result.

\begin{theorem}\label{oteorema}
    Let $\mathcal{F}$ be a codimension one distribution of degree $d$ on $\p$, such that $\Tf$ is nonsplit and nonstable with maximal order of nonstability. Then $d \geq 3$ and
    \begin{enumerate}
        \item $\Tf$ has Chern classes $(2-d,1,d)$, generated by a line, \(\textnormal{Sing}_1(\mathcal{F})\) is a curve of degree \(d^2+1\) with genus \(p_a = d(d-1)^2\) and \(\textnormal{Sing}_0(\mathcal{F})\) consists of \(d\) collinear points; or

        \item $\Tf$ has Chern classes $(2-d,2,2d)$, generated by a pair of skew lines (or a double line of genus -1), \(\textnormal{Sing}_1(\mathcal{F})\) is a curve of degree \(d^2\) (complete intersection) and \(\textnormal{Sing}_0(\mathcal{F})\) consists of two pairs of \(d\) collinear points (or \(d\) collinear points on the double line).
    \end{enumerate}
\end{theorem}

%\begin{theorem}
%    Let $\mathcal{G}$ be a foliation by curves of degree 1 on $\p$. Then
%    \begin{enumerate}
%        \item $\Ng^{\vee}$ is a $\mu$-stable rank 2 reflexive sheaf with $(c_1,c_2,c_3) = (-4,6,4)$, with $\textnormal{Sing}(\mathcal{G})$ a 0-dimension scheme of length 4;

%        \item $\Ng^{\vee}$ is a strictly $\mu$-semistable reflexive sheaf with $(c_1,c_2,c_3) = (-4,5,2)$ and given by an extension 
%        \begin{equation*}
%            0 \longrightarrow \op(-2) \longrightarrow \Ng^{\vee} \longrightarrow \mathcal{I}_L(-2) \longrightarrow 0,
%        \end{equation*}
%        where $L$ is a line. In this case, $\textnormal{Sing}(\mathcal{G}) = L' \cup p$ where $L'$ is a line and $p$ is 0-dimensional scheme of length 2, in particular $L \neq L'$;
        
%        \item $\Ng^{\vee} = \op(-2) \oplus \op(-2)$, and $\textnormal{Sing}(\mathcal{G})$ consists of either 2 skew lines or a double line of genus $-1$. Here $\textnormal{Sing}(\mathcal{G})$ is either 2 skew lines or a double line of genus 1.
%    \end{enumerate}
%\end{theorem}

\paragraph{}
\textbf{First Family with Maximal Order. }Here we will discuss the existence and properties of the distributions on the first family. Let \(\mathcal{F}\) be a codimension one distribution on the first family with maximal order of nonstability
\begin{equation*}
    0 \longrightarrow \Tf \longrightarrow \tpb \longrightarrow \mathcal{I}_Z(d+2) \longrightarrow 0.
\end{equation*}
So \(\Tf\) has Chern classes \((2-d,1,d)\) and is generated by an extension of a line. Consider \(F\) a rank 2 reflexive reflexive sheaf generated by the extension of a line \(L \subset \p\)
\begin{equation}\label{seq fam 1}
    0 \longrightarrow \op \longrightarrow F \longrightarrow \mathcal{I}_L(2-d) \longrightarrow 0.
\end{equation}
Now we may consider the exact sequence of the line
\begin{equation*}
    0 \longrightarrow \op(-2) \longrightarrow \op(-1)^{\oplus 2} \longrightarrow \mathcal{I}_L \longrightarrow 0,
\end{equation*}
and using the cohomology of the invertible sheaves on \(\p\) we obtain the vanishings:
\begin{equation*}
    H^1(\mathcal{I}_L) = H^2(\mathcal{I}_L(-1)) = 0.
\end{equation*}
Finally, using the above vanishings and Serre duality, we have
\begin{equation*}
    H^i(F(d-1-i)) = 0 \ \ \textnormal{for} \ \  i>0.
\end{equation*}
Therefore, by Corollaty \ref{regular} every rank 2 reflexive sheaf generated as (\ref{seq fam 1}) will be the tangent sheaf of a codimension one distribution of degree \(d\). 

We may calculate  $\Ext^1(\mathcal{I}_L(2-d),\op)$ by the exact sequence of the line and obtain that it has dimension $d+1$ and since the grassmaniann parametrizing lines on \(\p\) has dimension 4, there is a variety \(\mathbb{V}\) of dimension \(d +4\) that parametrize isomorphism classes of sheaves generated as (\ref{seq fam 1}). 

%Our goal is to calculate the moduli spaces $D(d,1,d)$. We can show irreducibility using the isomorphism
%\begin{equation*}
%    \Hom(F,\tpb) \simeq H^0(F^* \otimes \tpb) \simeq H^0(F(d-2) \otimes \tpb).
%\end{equation*}
%The twisted Euler sequence
%\begin{equation*}
%    0 \longrightarrow F(d-2) \longrightarrow F(d-1)^{\oplus 4} \longrightarrow F(d-2) \otimes \tpb \longrightarrow 0,
%\end{equation*}
%shows that
%\begin{align*}
%    h^0(F(d-2) \otimes \tpb) &= 4h^0(F(d-1)) - h^0(F(d-2)) \\
%    &= 4h^0(\op(d-1)) + 4h^0(\mathcal{I}_L(1)) - h^0(\op(d-2)) \\
%    &= 4 \binom{d+2}{3} + 8 - \binom{d+1}{3} \\
%    &= \frac{1}{2}d(d+1)(d+3) + 8.
%\end{align*}
%We obtain,
%\begin{equation*}
%    \dim(\Hom(F,\tpb)) = \frac{1}{2}d(d+1)(d+3) + 8, \ \ \forall \ F \ \textnormal{given by} \ (\ref{seq fam 1}). \triangle
%\end{equation*}

Returning to the distribution \(\mathcal{F}\), we have that $H^1_*(\Tf) = 0$ and therefore $Z = \textnormal{Sing}(\mathcal{F})$ is not contained in a hypersurface of degree $d$. Also, $\textnormal{Sing}_1(\mathcal{F}) = C$ have the following invariants
\begin{equation*}
    \deg(C) = d^2 + 1 \ \ \textnormal{and} \ \ p_a(C) = d(d-1)^2.
\end{equation*}
A possible example for \(C\) would be the union of a complete intersection of degree \(d^2\) with a line intersecting it in \(d\) distinct points.

\begin{example}
    Consider the codimension one distribution \(\mathcal{F}\) of degree 3 defined by de 1-form in homogeneous coordinates \((x,y,z,w)\):
    \begin{align*}
        \omega &= (xy^2z - xyz^2 + yw^3 - zw^3)dx \\
        &+ (-x^2yz + xyz^2 - 2x^3w - 2y^3w + 3xyzw - 2z^3w -xw^3 + zw^3 + 3w^4)dy \\
        &+ (x^2yz - xy^2z + xw^3 - yw^3)dz \\
        &+ (2x^3y + 2y^4 + -3xy^2z + 2yz^3 - 3yw^3)dw.
    \end{align*}
    Then \(\textnormal{Sing}_0(\mathcal{F})\) has length 3 and \(\textnormal{Sing}_1(\mathcal{F})\) is the union of the complete intersection defined by the ideal \((x^3+y^3+z^3, xyz + w^3 )\) with the line \(L = \{y = w = 0\}\).
\end{example}

%\begin{equation*}
%    \textcolor{red}{\Hom(\Tf,\Nf) = \Hom(F,\mathcal{I}_Z(d+2)) = ?} (\textnormal{Tangent Space})
%\end{equation*}
%\begin{equation*}
%    \textcolor{red}{\Ext^1(\Tf,\Nf) = \Ext^1(F,\mathcal{I}_Z(d+2)) = ?}(\textnormal{Obstruction}) 
%\end{equation*}

%\begin{align*}
%    \Hom(F,F) &\leq \Hom(F,\op) + \Hom(F,\mathcal{I}_L(2-d))\\
%    &\leq h^0(F(d-2)) + \Ext^1(F,\op(-d)) \\
%    &\leq h^0(F(d-2)) + h^2(F(d-4))
%\end{align*}

%Using the classification of distributions with degree $\leq 2$, we may describe the low degree cases
%\begin{enumerate}
%    \item For $d = 0$, we will have that $\Tf$ has invariants $(2,1,0)$ and will be forced to be split, so
%    \begin{equation*}
%        \Tf = \op(1) \oplus \op(1)
%    \end{equation*}

%    \item For $d = 1$, we will have $\Tf$ with $(1,1,1)$, stable and $\textnormal{Sing}_1(\mathcal{F})$ will be a conic.

%    \item For $d = 2$, we will have $\Tf$ with $(0,1,2)$, strictly $\mu$-semistable and $\textnormal{Sing}_1(\mathcal{F})$ is going to be a ACM quintic of genus 2. Among these distributions we have logarithmic foliations $\mathcal{L}(1,1,2)$.
%\end{enumerate}
\paragraph{}
\textbf{Second Family with Maximal Order. }Now let's study the second family on Theorem \ref{oteorema} given by
\begin{equation*}
    0 \longrightarrow \Tf \longrightarrow \tpb \longrightarrow \mathcal{I}_Z(d+2) \longrightarrow 0,
\end{equation*}
where $\Tf$ has Chern classes \((2-d,2,2d)\). Let \(F\) be a rank 2 reflexive sheaf generated as an extension
\begin{equation}\label{seq def 2}
    0 \longrightarrow \op \longrightarrow F \longrightarrow \mathcal{I}_Y(2-d) \longrightarrow 0,
\end{equation}
where $Y$ is a pair of skew lines (or a double line of genus -1). Since \(H^0(\mathcal{I}_Y(1)) = 0\), the sheaf \(F(d-1)\) cannot be globally generated and is not guarantee that every sheaf generated by this extension is the tangent sheaf of a codimension one distribution.

%\paragraph{}
%\(\triangle\) If $Y$ is the disjoint union of two lines $L_1$ and $L_2$, therefore
%\begin{equation*}
%    \omega_Y \simeq \omega_{L_1} \oplus \omega_{L_2} \simeq \mathcal{O}_{L_1}(-2) \oplus \mathcal{O}_{L_2}(-2)
%\end{equation*}
%so we may calculate the extension of the second family,
%\begin{equation*}
%    H^0(\omega_Y(4-c_1)) = H^0(\omega_Y(d+2)) = H^0(\mathcal{O}_{L_1}(d)) \oplus H^0(\mathcal{O}_{L_2}(d)).
%\end{equation*}
%Using that $h^0(\mathcal{O}_L(d)) = d+1$, for a line $L$, then
%\begin{equation*}
%    \dim(\Ext^1(\mathcal{I}_Y(2-d),\op)) = 2d +2.
%\end{equation*}
%Skew lines are parametrized by a 6 dimensional variety, and so we have a variety of dimension $2d+7$ parameterizing the sheaves in the second family. For $D(d,2,2d)$, The twisted Euler sequence
%\begin{equation*}
%    0 \longrightarrow F(d-2) \longrightarrow F(d-1)^{\oplus 4} \longrightarrow F(d-2) \otimes \tpb \longrightarrow 0,
%\end{equation*}
%shows that
%\begin{align*}
%    h^0(F(d-2) \otimes \tpb) &= 4h^0(F(d-1)) + h^1(F(d-2)) - h^0(F(d-2)) \\
%    &= 4h^0(\op(d-1)) + h^1(\mathcal{I}_Y) - h^0(\op(d-2)) \\
%    &= 4 \binom{d+2}{3} + 1 - \binom{d+1}{3} \\
%    &= \frac{1}{2}d(d+1)(d+3) + 1.
%\end{align*}
%We obtain,
%\begin{equation*}
%    \dim(\Hom(F,\tpb)) = \frac{1}{2}d(d+1)(d+3) + 1, \ \ \forall \ F \ \textnormal{given by} \ (\ref{seq def 2}). \triangle
%\end{equation*}
We may calculate that \(\Ext^1(\mathcal{I}_Y(2-d),\op)\) has dimension \(2d + 2\). Since skew lines are parametrized by the symmetric product of the grassmaniann parametrizing lines on \(\p\) without the diagonal, they are parametrized by a variety of dimension 6. So there is variety of dimension \(2d+7\) parametrizing isomorphism classes of sheaves generated as (\ref{seq def 2}).

The distribution \(\mathcal{F}\) has $\textnormal{Sing}_1(\mathcal{F}) = C$ a complete intersection of degree $d^2$. Since the only nonzero first cohomology group of \(\Tf\) is \(h^1(\Tf(d-2)) = 1\), which implies that \(H^1(\Tf(-2)) = 0\) and we obtain that \(\textnormal{Sing}(\mathcal{F})\) is not contained in a hypersurface of degree \(d\). We cannot guarantee the existence for every sheaf generated as (\ref{seq def 2}), but we may construct an explicitly example vanishing where we want.

\begin{example}
    Consider the codimension one distribution \(\mathcal{F}\) of degree 3 defined by the 1-form in homogeneous coordinates
    \begin{align*}
        \omega &= (x^3y +y^4- x^3z + xy^2z - y^3z + xyz^2 + yz^3 - z^4 - 2xyzw + yw^3 + zw^3 - 2w^4)dx \\
        &+ (-x^4 - xy^3 + x^3z - x^2yz + y^3z + 3xyz^2 - xz^3 \\
        &+ z^4 - 2x^3w - 2y^3w + xyzw -2z^3w - xw^3 + 3zw^3 + w^4)dy \\
        &+ (x^4 - x^3y + xy^3 - y^4 - x^2yz - 3xy^2z + xz^3 \\
        &- yz^3 + 2x^3w + 2y^3w + xyzw + 2z^3w - xw^3 - 3yw^3 + w^4)dz \\
        &+ (2x^3y + 2y^4 - 2x^3z + 2x^2yz - xy^2z - 2y^3z - xyz^2 + 2yz^3 - 2z^4 + 2xw^3 - yw^3 - zw^3)dw.
    \end{align*}
    Then \(\textnormal{Sing}_0(\mathcal{F})\) has length 6 and \(\textnormal{Sing}_1(\mathcal{F})\) is the complete intersection of degree 9 defined by the ideal \((x^3+y^3+z^3, xyz + w^3 )\).
\end{example}

\paragraph{}
Notice that in our classification the only locally free sheaves that appear are split or the nullcorrelation bundle in Example \ref{nullcorrelation}, which is stable and does not define a codimension one foliation. We may construct an example of a codimension one distribution with locally free tangent sheaf that is nonsplit and unstable. 
\begin{example}
    Take \(L \subset \p\) a line and \(Y\) a double structure on \(L\) with genus \(p_a = -4\), for multiple structures on a line see \cite[Proposition 1.4]{nollet1997hilbert}. Now consider the extension
\begin{equation}\label{fibradounstable}
    0 \longrightarrow \op \longrightarrow F \longrightarrow \mathcal{I}_Y(-1) \longrightarrow 0.
\end{equation}
Then \(F\) is an unstable rank 2 locally free sheaf with Chern classes \((c_1,c_2) = (-1,2)\). Using Serre Duality, we obtain
\begin{equation*}
    H^3(F(m))^{\vee} = H^0(F(-m-3)), \ \ \textnormal{for all } m \in \Z .
\end{equation*}
So we have the vanishings
\begin{equation*}
    H^3(F(m)) = 0, \ \ \textnormal{for } m\geq-2.
\end{equation*}
The exact sequence 
\begin{equation*}
    0 \longrightarrow \mathcal{O}_L(3) \longrightarrow \mathcal{O}_Y \longrightarrow \mathcal{O}_L \longrightarrow 0
\end{equation*}
together with the structural exact sequence gives us
\begin{align*}
    H^1(\mathcal{I}_Y(m)) &= 0, \ \ \textnormal{for } m\geq 2; \\
    H^2(\mathcal{I}_Y(m)) &= 0, \ \ \textnormal{for } m \geq -1.
\end{align*}
This vanishings with the exact sequence (\ref{fibradounstable})
\begin{align*}
    H^1(F(m)) &= 0, \ \ \textnormal{for } m\geq 3; \\
    H^2(F(m)) &= 0, \ \ \textnormal{for } m \geq 0.
\end{align*}
All this vanishings on the cohomology of \(F\) gives us the identity
\begin{equation*}
    H^i(F(4-i)) = 0, \ \ \textnormal{for } i >0.
\end{equation*}
So \(F\) is 4-regular and therefore \(F(4)\) is globally generated. Applying the Bertini type Theorem \cite[Corollary A.4]{CCJ}, we obtain that there exists a codimension one distribution \(\mathcal{F}\) of degree 7 such that
\begin{equation*}
    \Tf \simeq (F(4))^{\vee}(1) \simeq F(-2).
\end{equation*}
Then \(\Tf\) is an unstable locally free sheaf that is nonsplit and it has \(t_{\mathcal{F}} = 2\).
\end{example}

\section{Final Remarks}\label{section 5}

%\subsection{Logarithmic Foliations}
\paragraph{}
Logarithmic foliations form a well-know familiy of codimension one foliations, they form a irreducible component in every space of codimension one foliations, see \cite{calvo1994irreducible}. A codimension one distribution \(\mathcal{F}\) of degree \(d\) on \(\p\) is called a logarithmic foliation of type \((d_1.\dots,d_r)\) if is given by a 1-form 
\begin{equation}\label{1-form log}
    \omega = f_1 \dots f_r \sum_{i=1}^r \lambda_i \frac{df_i}{f_i}.
\end{equation}
such that
\begin{itemize}
            \item \(1 \leq d_1 \leq \dots \leq d_r\);
            
            \item \(\lambda_i \in \C\) and \(\sum \lambda_i d_i = 0\);

            \item \(f_i\) is a homogeneous polynomial of degree \(d_i\) and \(\sum d_i = d+2\).
        \end{itemize}
We say that \(\mathcal{F}\) is a generic logarithmic foliation if in addition we have that \(\lambda_i \neq 0\) for every \(i\) and each \(D_i := \{ f_i = 0\}\) is in general position, i.e, intersects each other transversely. In \cite[Theorem 3.]{cukierman2006singularities} the authors proved that if \(\mathcal{F}\) is a generic logarithmic foliation given by (\ref{1-form log}), then
\begin{equation*}
    \textnormal{Sing}(\mathcal{F}) = C \cup U.
\end{equation*}
Where \(U\) is the zero dimensional component with length
\begin{equation*}
    N(3,h) = \textnormal{coefficient of} \ h^3 \ \textnormal{in} \ \ \frac{(1-h)^4}{(1-d_1h)\dots(1-d_rh)}
\end{equation*}
and \(C\) is the one dimensional component given by
\begin{equation*}
    C = \bigcup_{i < j}(D_i \cap D_j).
\end{equation*}

With this theorem we obtain an explicit formula for the invariants of the singularities of such foliations 
\begin{equation*}
    \deg(C) = \sum_{i < j}d_id_j,
\end{equation*}
and we may compare with the ones obtained in Theorem \ref{oteorema}.

\begin{lemma}
    There are no generic logarithmic foliations among the two families with maximal order of nonstability.
\end{lemma}
\begin{proof}
    Let \(\mathcal{F}\) be a generic logarithmic foliation of type \((d_1,\dots,d_r)\) and denote
    \(C\) 1-dimensional component of \(\mathcal{F}\).
    Then we have 
    \begin{equation*}
        \deg(C) = \sum_{i<j}d_id_j \leq \frac{r(r-1)}{2r^2}(d+1)^2.
    \end{equation*}
    But notices that
    \begin{equation*}
        \frac{r(r-1)}{2r^2}(d+1)^2 < \frac{1}{2}(d+1)^2 < d^2 + i, \ \ \textnormal{for } d\geq3,
    \end{equation*}
    where \(i \in \{0,1\}\). This means that \(\deg(C)\) cannot attain the value of the singularities in Theorem \ref{oteorema} and therefore the result follows.
\end{proof}

\paragraph{}
\textbf{Acknowledgments.} The author would like to thank his advisors Alan Muniz and Marcos Jardim for insightful discussions on this work. This project was finished at the University of Bari, under the supervision of Mauricio Côrrea, which the author is grateful for the hospitality and support. He is supported by a Phd scholarship of Coordenação de Aperfeiçoamento de Pessoal de Nível Superior - Brasil (CAPES) - Finance Code 001, as well as by the CAPES–PDSE Program.

\noindent (Pedro Pfarrius Barbassa) IMECC-UNICAMP, Departamento de Matemática,  
 R. Sérgio Buarque de Holanda, 651 - Cidade Universitária, Campinas - SP, 13083-859, Brasil.
 
\noindent\textit{Email address:} \texttt{p185839@dac.unicamp.br}
\end{document}